\documentclass{amsart}
\usepackage{amssymb}
\usepackage{amsmath}
\evensidemargin\oddsidemargin
\makeatletter
\@addtoreset{equation}{section}

\makeatother
\newtheorem{theo}{Theorem}

\newtheorem{lem}{Lemma}
\newtheorem{cor}{Corollary}
\newcommand{\PP}{\Bbb{P}}

\newcommand{\R}{\Bbb{R}}

\newcommand{\E}{\Bbb{E}}
\newcommand{\V}{\mbox{Var}}

\newcommand{\be}{\begin{equation}}
\newcommand{\ee}{\end{equation}}
\newcommand{\fc}{\mathcal{F}}
\numberwithin{theo}{section}
\numberwithin{equation}{section}
\numberwithin{propo}{section}
\numberwithin{lem}{section}
\numberwithin{cor}{section}

\begin{document}
\date{}
\title{LIL type behavior of multivariate L\'evy processes at zero}
\author[U. Einmahl]{ Uwe Einmahl}
\address{Department of Mathematics, Vrije Universiteit Brussel\\
Pleinlaan 2, B-1050 Brussels, Belgium}
\email{ueinmahl@vub.be}

\begin{abstract} We study the almost sure behavior of suitably normalized multivariate L\'evy processes as $t \downarrow 0.$ Among other results we find necessary and sufficient conditions for a  law of a very slowly varying function  which includes a general law of the iterated logarithm in this setting.  We also look at the corresponding cluster set problem.

\end{abstract}

\maketitle


\section{Introduction}
Let $\{X_t: t \ge 0\}$ be a $d$-dimensional L\'evy process with $X_0=0$ and characteristic triplet $(\gamma, \Sigma, \Pi)$, where $\gamma \in \R^d$ and $\Sigma$ is a symmetric, non-negative definite $d \times d$ matrix.  $\Pi$ is the L\'evy measure which is a measure defined on the $\sigma$-algebra of all $d$-dimensional Borel subsets of $\R^d$ satisfying $\Pi(\{0\})=0$  and
\be \label{1}
\int (1 \wedge |y|^2) \Pi(dy) < \infty,
\ee
where $| \cdot |$ will  always denote the Euclidean norm on $\R^d.$\\
Moreover, if we set $\overline{\Pi}(x) = \Pi\{y: |y|>x\}, x > 0,$ condition (\ref{1}) can be also written as
\be \label{2}
\int_0^1 \overline{\Pi}(\sqrt{t}) dt < \infty.
\ee
In this paper we are interested in the almost sure behavior of  suitably normalized  L\'evy processes as $t \downarrow 0.$ Our starting point is the following $d$-dimensional version of the law of the iterated logarithm for L\'evy processes at zero which states that with probability one,
$$\limsup_{t \downarrow 0} \frac{|X_t|}{\sqrt{2 t\log \log 1/t}} = \sigma,$$
where $\sigma^2$ is the largest eigenvalue of the matrix $\Sigma.$\\
This follows easily from the 1-dimensional case (see Proposition 47.11 in \cite{Sato}). To see that just write $(X_t)_{t \ge 0}$ as a sum of a Gaussian process and a jump process which is possible by the L\'evy-It\^o decomposition (see, for instance, Theorem 1 on p. 13 in \cite{Ber0}). Then applying the 1-dimensional result for the $d$ components of the jump process we see that this process is of almost sure order $o(\sqrt{t\log \log 1/t})$ as $t \downarrow 0$ and the above result follows from the LIL for $d$-dimensional Brownian motion.
So if $\Sigma$ is non-trivial the almost sure behavior of the L\'evy process is completely determined by its Gaussian part. \\
If we have a purely non-Gaussian L\'evy process, that is, if $\Sigma$ is the zero-matrix,  the above $\limsup$ is equal to 0 and it is natural to ask whether one can  find a different (and necessarily smaller) function $b(t), 0 \le t \le 1$ in this case such that with probability one,
\be \label{Frage}
0 < \limsup_{t \downarrow 0} \frac{|X_t|}{b(t)} < \infty.
\ee
We speak in this case  of LIL behavior.
This problem has been studied in dimension 1  and let us give a short summary of what is already known in this case: The first result is classical and it is  due to Khintchine (see Proposition 47.13 in \cite{Sato}). It states that for any positive, continuous and increasing function $g$ satisfying $g(t)/\sqrt{t\log\log 1/t} \to 0$ as $ t\to 0$ there exists a 1-dimensional L\'evy process such that with probability one,
$$\limsup_{t \downarrow 0} \frac{|X(t)|}{g(t)} = \infty.$$
This  shows that  the above function $b(t)$ provided that it exists can be arbitrarily close 
to $\sqrt{t\log\log 1/t}$.  Fristedt \cite{Fri} found an example where one has  for $\beta > 0$ with probability one,
$$\limsup_{t \downarrow 0}\frac{|X_t|}{\sqrt{t}{(\log\log 1/t)}^{(1-\beta)/2}} =\sqrt{2}.$$
Note that if we choose $\beta=1$ we get the normalizer $\sqrt{t}.$\\  Bertoin, Doney and Maller  \cite{Ber}  obtained a complete result in this case by finding a necessary and sufficient condition for a 1-dimensional L\'evy process to satisfy with probability one,
$$\limsup_{t \downarrow 0} \frac{|X_t|}{\sqrt{t}} = \lambda.$$
This condition is in terms of an integral test which also allows to determine the constant $\lambda$. The authors  give examples where $\lambda$ is finite and positive, but it is also possible that it is zero or infinity.   Later this result was extended  to a functional limit theorem (see \cite{Buch}).\\
The next  step was done in Savov \cite{Sav} where   the author found an extension of the integral test in \cite{Ber}  to more general  functions $b(t), 0 \le t \le 1$. He also provided a method for calculating a possible normalizing function $b$ in terms of the L\'evy measure $\Pi$ which is related to the well known LIL of Klass \cite{Kla} in the random walk case. As in this classical case, the $\limsup$ results for this general normalizing function require  extra integrability conditions  and consequently there are certain cases where the general normalizing sequence cannot be used and one has to rely on other methods for finding a suitable normalizing function. (See Proposition 3.1 in \cite{Sav} for an interesting example.)  Finally, Savov \cite{Sav} also indicated a possible link with the paper \cite{EL-1} where LIL  type results for the random walk in the infinite variance case are considered.  \\[.1cm]
Given  the work in \cite{EL-1, EL} it appears now very natural to ask
 whether and when one can find ``nice'' functions $b$  such that (\ref{Frage}) holds. In view of the results in \cite{EL-1} where among other things a ``law of a very slowly varying function" has been proven and the afore-mentioned results of \cite{Fri} and \cite{Ber} one could simply ask when one has with probability one,
$$0 < \limsup_{t \downarrow 0}h(1/t) \frac{|X_t|}{\sqrt{t \log \log 1/t} }<\infty,$$
 where $h: [0,\infty[ \to ]0,\infty[$ is non-decreasing and slowly varying at infinity. \\
 Another interesting question is finally whether one can establish analogous results in the $d$-dimensional case. Both questions will be addressed in the present paper.
\section{Statement of Main Results}
Unless otherwise indicated we assume from now on that $\{X_t: t \ge 0\}$ is a purely non-Gaussian $d$-dimensional L\'evy process with $X_0=0$. Thus,   the matrix $\Sigma$ in the characteristic triplet $(\gamma, \Sigma, \Pi)$ is equal to the zero-matrix. Furthermore by a standard argument we can ignore the ``big jumps''  (see, for instance, \cite{Ber, Sav}) so that  we can assume  that the L\'evy measure $\Pi$ is supported by the unit ball $D$ in $\R^d.$\\
Thus, $X_t$  has characteristic function $\theta \mapsto \E \exp(i\langle \theta, X_t \rangle) = \exp(t \Psi(\theta)),$ where
$$\Psi(\theta) = i \langle \gamma, \theta \rangle + \int_D \left(e^{i\langle \theta, y\rangle} - 1 - i\langle \theta, y\rangle\right)\Pi(dy), \theta \in \R^d.$$
Next we define a function $V(t), t \ge 0$ via the L\'evy measure $\Pi$ as follows,
$$V(t)=\sup_{|z| \le 1} \int_{|y| \le t} \langle y, z \rangle^2 \Pi(dy), t \ge 0. $$
Note that
$$V(t) \le \int_{|y| \le t} |y|^2 \Pi(dy), t > 0$$
Recalling (\ref{1}) we see  via the dominated convergence theorem that 
$$V(t) \searrow 0  \mbox{ as }t \downarrow 0.$$
To formulate our first results we still have to introduce some function classes.
As in \cite{EL-1} we denote the class of the continuous and non-decreasing slowly varying functions $h: [0,\infty[ \to [0,\infty[$ by $\mathcal{H}_1$ and we further look at subclasses $\mathcal{H}_q, 0 \le q <1,$
consisting of all functions $h$  satisfying the condition
$$h(xf_{\tau}(x))/h(x) \to 1 \mbox{ as } x \to \infty, 0 \le \tau < 1-q,$$
where $f_{\tau}(x):= \exp((\log x)^{\tau}), x \ge 1.$\\
We call the functions in $\mathcal{H}_0$ also  ``very slowly varying''. Examples for such functions are the functions $x \mapsto (\log \log x)^{\alpha}$ and $t \mapsto (\log x)^{\alpha}, x \ge e^e,$ where $\alpha > 0.$\\
Our first result gives an upper bound for  $\limsup_{t \downarrow 0} h(1/t) |X_t|/\sqrt{2t\log \log 1/t}$ if $h$ is slowly varying at infinity.
\begin{theo} \label{upper}
Let $\{X_t: t \ge 0\}$ be a purely non-Gaussian d-dimensional L\'evy process and suppose that $b(t)=\sqrt{t \log \log 1/t}/h(1/t)$ for small $t$, where $h \in \mathcal{H}_1$.    Assume that
\be \label{3}
\int_0^1 \overline{\Pi} (b(t))dt < \infty
\ee
and for some $\lambda \ge 0,$
\be \label{3a}
\limsup_{t \downarrow 0}{ V(b(t)) h^2(1/t)} \le \lambda^2/2.
\ee
 Then we have with probability one:
 $$\limsup_{t \downarrow 0} h(1/t) \frac{|X_t|}{\sqrt{t\log \log 1/t}} \le \lambda $$
\end{theo}
The corresponding lower bound result is as follows,
\begin{theo} \label{lower}
Let $\{X_t: t \ge 0\}$ be a purely non-Gaussian  d-dimensional L\'evy process. Let $\lambda \ge 0$, $h \in \mathcal{H}_q$ and let $b$ be as in Theorem \ref{upper}.  Assume that
\be \label{3b}
\limsup_{t \downarrow 0}{ V(b(t)) h^2(1/t)} \ge \lambda^2/2.
\ee
 Then we have with probability one,
 $$\limsup_{t \downarrow 0} h(1/t) \frac{|X_t|}{\sqrt{t\log \log 1/t}} \ge  (1-q)^{1/2}\lambda $$
\end{theo}
Combining the two above results we get the following  result which we could call the law of a very slowly varying function for  L\'evy processes.
\begin{cor}\label{slow}
 Let $\{X_t: t \ge 0\}$ be a purely non-Gaussian d-dimensional L\'evy process.  Suppose that 
 $b(t)=\sqrt{t \log \log 1/t}/h(1/t)$ for small $t$, where $h \in \mathcal{H}_0$.  Assume that condition (\ref{3}) is satisfied.
 If $\lambda \ge 0$, the following are equivalent,
 \begin{itemize}
\item [(a)] $$\limsup_{t \downarrow 0} h(1/t)\frac{|X(t)|}{\sqrt{t\log \log 1/t}}= \lambda \;\;\mbox{ with  probability 1 }$$
\item[(b)] $$\limsup_{t \downarrow 0} V(b(t))h^2(1/t) =\lambda^2/2.$$
\end{itemize}
 
\end{cor}
Condition (\ref{3}) is not required if $h(x) = O(\sqrt{\log \log x})$. In this case it  already follows from (\ref{2}). So if we choose $h(x) =\sqrt{\log \log x}, x \ge e,$ we get:
$$\limsup_{t \downarrow 0} \frac{|X_t|}{\sqrt{t}} = \lambda \mbox{ with prob. 1} \Longleftrightarrow \limsup_{t \downarrow 0} V(t)\log \log 1/t = \lambda^2/2.$$
Also note that we can choose  functions $h(x)$ which converge extremely slowly to infinity as $x \to \infty.$ This gives us, as in the classical result of Khintchine,   normalizers $b(t)$ which are very close to $\sqrt{t\log \log 1/t}$ and this will happen if $V(t)$ converges extremely slowly to $0$ as $t \to 0$.\\ [.1cm]
The next lemma shows that  condition  (\ref{3})  is actually redundant for any function  $h$ satisfying  condition (\ref{xy}) below (and not only for functions of order $O(\sqrt{\log \log x}).$) 
\begin{lem} \label{extra-lemma}
 Let $h \in \mathcal{H}_0$  be a function such that for some $x_0 \ge e^e$ and  $x \ge x_0$,
\be  \label{xy}
 \exists\, \vartheta \in ]0,1[: h(x) \le \exp((\log \log x)^{\vartheta})
\ee
Then  we have if $b(t)=\sqrt{t \log \log 1/t}/h(1/t)$ for small $t,$
\be
V(b(t))= O(h^{-2}(1/t)) \mbox{ as }t \downarrow 0 \Longrightarrow \int_0^{1} \overline{\Pi}(b(t)) dt < \infty.
\ee
\end{lem}
Condition (\ref{xy}) is sharp. The assertion of Lemma \ref{extra-lemma} is no longer true if $\vartheta=1.$ 
One can find examples where one has $V(t) \sim (\log 1/t)^{-2}$ as $t \to 0$ and, at the same time,   $\int_0^{e^{-2}} \overline{\Pi}(\sqrt{t \log \log 1/t}/\log 1/t) dt =\infty.$ (See, for instance, Example 2 in \cite{Fri}.)
So we cannot apply Lemma \ref{extra-lemma} if $h(x)=\log x, x \ge 1.$\\[.2cm] Condition (\ref{xy}) is satisfied for all functions $h(x)=(\log \log x)^{q}, x >e,$ where  $q>0.$
We arrive at the following general LIL  for  L\'evy processes from which one can easily re-obtain the afore-mentioned result of \cite{Fri} and find many other examples.
\begin{cor} \label{LIL}  Let $\{X_t: t \ge 0\}$ be a purely non-Gaussian d-dimensional L\'evy process. 
Given any  $- \infty < p < 1/2$ and any $\lambda \ge 0$, the following are equivalent:
\begin{itemize}
\item [(a)] $$\limsup_{t \downarrow 0} \frac{|X(t)|}{\sqrt{t}(\log \log 1/t)^{p}}= \lambda \;\;\mbox{ with probability 1}$$
\item[(b)] $$\limsup_{t \downarrow 0} V(t) (\log \log 1/t)^{1 -2p} =\lambda^2/2.$$
\end{itemize}
\end{cor}
An important tool for proving these results will be a general result on the almost sure behavior of normalized $d$-dimensional L\'evy processes which extends Theorem 2.1 in \cite{Sav} to this more general setting. We weaken  assumption (2.2) in this paper slightly by still assuming that $b(t)/t$ converges to infinity as $t \to 0$, but we do not require monotonicity.  Our condition (\ref{50}) holds in this case as well, but also if $b(t)=\sqrt{t\log \log 1/t}/h(1/t), t > 0,$ where $h: [0, \infty[ \to ]0,\infty[$ can be any function which is slowly varying at infinity. This easily follows from the Karamata representation for slowly varying functions (see, for instance, \cite{Ber0}, page 9). 
\begin{theo} \label{dimension_d}
Let $\{X_t: t \ge 0\}$ be a purely non-Gaussian d-dimensional L\'evy process and let $b(t), 0 \le t \le 1$ be a continuous and increasing real-valued function such that $b(t)/t \to \infty$ as $t \to 0.$  Assume also that the following two conditions hold on a suitable interval $[0, t_0]$:
\begin{eqnarray}
b(t)/t^{\rho} \mbox{ is non-decreasing for some }\rho >1/3 \label{49}\\
\forall \epsilon > 0\; \exists\, 0 < \delta_{\epsilon} < t_0:\; b(s)/b(t) \ge (1-\epsilon) s/t, 0 \le s \le t \le \delta_{\epsilon} \label{50}
\end{eqnarray}
Under condition (\ref{3}) we have with probability one, 
$$\limsup_{t \downarrow 0} \frac{|X_t|}{b(t)} = \alpha_0,$$
where
$$\alpha_0:=\sup\left\{\alpha \ge 0: \int_0^1 \frac{1}{t} \;\exp\left(-\frac{\alpha^2 b^2(t)}{2t V(b(t))}\right) dt = \infty\right\} .$$
 \end{theo}
We mention that Theorem \ref{dimension_d} and all the previous results remain true if we replace the function $V(t)$ by the larger function
\[V_1(t):=\sup_{|z| =1} \int_{|\langle y, z \rangle| \le t} \langle y, z \rangle^2 \Pi(dy).\]
We will prove this at the end of Section 3.
This function plays an important role  in the weak convergence theory for matrix normalized L\'evy processes (see  \cite{MM}).  \\[.1cm]
We now turn to the cluster set in Theorem \ref{dimension_d}, that is, the set of all limit points of $X(t_n)/b(t_n)$ for sequences $t_n \downarrow 0$. We denote this set by $$C(\{X_t/b(t): t \downarrow 0\}).$$  It is well known that this (random) set is  equal to a deterministic set $A \subset \R^d$ with probability one.  (For a proof of a  more general version of this fact the reader is referred to  Sect. 2  in \cite{Buch}.) 
\begin{theo} \label{star} 
 Assume that the conditions of Theorem \ref{dimension_d} are satisfied. If $\alpha_0 <\infty$  the deterministic cluster set $A=C(\{X_t/b(t): t \downarrow 0\})$ is compact,  symmetric about zero and star-like with respect to zero. Moreover, we have $\alpha_0=\sup_{x \in A} |x| $ and there exists a unit vector $z \in \R^d$ such that $\alpha_0 z \in A.$
\end{theo}
\begin{theo}\label{star-}
 Suppose that $0< \alpha_0 < \infty$  and  that $A$ is a subset of $\R^d$ as in Theorem \ref{star}. Let  $h \in \mathcal{H}_0$ be a function such that $\lim_{x \to \infty} h(x)=\infty$ and set
$b(t)=\sqrt{t \log \log 1/t}/h(1/t), t \le e^{-e}.$   There exists a purely non-Gaussian $d$-dimensional L\'evy process $\{X_t: t \ge 0\}$ such that we have  with probability one, 
$\limsup_{t\downarrow 0} |X_t|/b(t) = \alpha_0 \mbox{ and }
C(\{ X_t/b(t): t \downarrow 0\})=A.$
\end{theo}
All the results in this section have counterparts for the $d$-dimensional random walk in the infinite second moment case (see \cite{EL} for Theorems 2.1-2.3). Cluster sets in the random walk case have been studied in \cite{E-1} and \cite{EK}  where it also has been shown that any bounded and closed set which is symmetric and star-like w.r.t. zero can occur as cluster set, but this is done in these two papers only for a very specific normalizing sequence. A new feature is here that Theorem \ref{star-} holds for a large class of normalizing sequences and we  give a somewhat easier proof since we can define a suitable L\'evy measure directly via a certain representation of the set $A$.\\
Theorem \ref{dimension_d} will be proven in Section 3. Our proof follows essentially the method developed in \cite{Ber} and  which was further refined in \cite{Sav}. A new element is that we use  exponential inequalities for sums of independent random vectors instead of Berry-Esseen type results. This should make it possible to extend our method  to the infinite-dimensional case should this be needed. In Section 4 we then show how Theorems \ref{upper} and  \ref{lower} follow from Theorem \ref{dimension_d} and we prove Lemma \ref{extra-lemma}. In  Section 5 we first provide  a general result on clustering (see Theorem \ref{Kesten2}) which is valid for any L\'evy process and we then derive a criterion for purely non-Gaussian L\'evy processes from it (see Lemma \ref{lemClu}). We finally prove Theorems \ref{star} and \ref{star-} in subsections 5.3 and 5.4.
\section{Proof of Theorem \ref{dimension_d}}
Recall that we assume that the L\'evy measure $\Pi$ is supported by the unit ball $D$ in $\R^d.$
Since any function $b$ we consider is continuous and increasing, its  inverse function $b^{\leftarrow}(t)$ is well defined for $0 \le t \le b(1)$. As we are only interested in the local behavior at zero, we can assume 
w.l.o.g. that $b(1)=1$  by redefining the function $b$ on a suitable interval $[s_0, 1]$, where $0 < s_0 < 1.$\\
By a standard argument from measure theory, we also have that condition (\ref{3}) holds if and only if
\be \label{inv-function}
\int_{0 < |y| \le 1} b^{\leftarrow}(|y|)\Pi(dy) < \infty.
\ee
We need the following lemma.
\begin{lem} \label{cond} Assume that $b: [0,1] \to [0,1]$ satisfies  conditions (\ref{49}) and (\ref{50}).  If $\Pi$ is a L\'evy measure supported by the unit ball $D \subset \mathbb{R}^d$ and if
 condition (\ref{3}) holds,  we have:
\begin{itemize}
\item [(a)] $$\int_0^1 \frac{1}{b(t)^3}\int_{0 < |y| \le b(t)} |y|^3 \Pi(dy) dt < \infty.$$
\item [(b)] $$\int_0^1 \overline{\Pi}(\epsilon b(t)) dt < \infty,\, 0 < \epsilon <1.$$
\item [(c)]  $$\frac{t}{b(t)}\int_{b(t) < |y| \le 1} |y| \Pi(dy) \to 0 \mbox{ as } t \to 0.$$

\end{itemize}
\end{lem}
\begin{proof} (a)
From (\ref{49}) it easily follows htat 
$$\frac{|y|}{b(t)} \le \frac{b^{\leftarrow}(|y|)^{\rho}}{t^{\rho}}, b^{\leftarrow}(|y|) \le t \le 1.$$
Since $\rho > 1/3$, we can conclude that
\begin{eqnarray*}
&&\int_0^{1} \frac{1}{b(t)^3}\int_{0 < |y| \le b(t)} |y|^3 \Pi(dy) dt 
= \int_{0 < |y| \le 1}\int_{b^{\leftarrow}(|y|)}^{1} \frac{1}{b(t)^3}dt |y|^3\Pi(dy)\\
&\le& \int_{0 < |y| \le 1}\int_{b^{\leftarrow}(|y|)}^{1}  t^{-3\rho} dt\, b^{\leftarrow}(|y|)^{3\rho} \Pi(dy)
\le\frac{1}{3 \rho-1}  \int_{0 < |y| \le 1} b^{\leftarrow}(|y|)\Pi(dy),
\end{eqnarray*}
where the last integral is finite by condition (\ref{inv-function}). Thus(a) holds.\\[.2cm]
(b) Note that
$$\overline{\Pi}(\epsilon b(t)) - \overline{\Pi}( b(t))=\int_{\epsilon b(t) \le  |y| < b(t)} \Pi(dy)
\le \epsilon^{-3} \frac{1}{b(t)^3}\int_{0 < |y| \le b(t)} |y|^3 \Pi(dy).$$
Combining this inequality with part (a), we see that (b) holds as well.\\[.2cm]
(c) Observe that
$$  t \left| \int_{b(t) < |y| \le 1} y \Pi(dy)\right|/ b(t) \le t  \int_{b(t) < |y| \le 1} |y| \Pi(dy)/b(t).$$
On account of condition (\ref{50})  we can find $0 < \overline{t} \le 1$  such that  
$$b(t_1)/b(t_2) \ge t_1/(2t_2), 0 \le t_1 \le t_2 \le \overline{t}.$$ 
It follows that
$$\frac{b(t)}{|y|} =\frac{b(t)}{b(b^{\leftarrow}(|y|)}\ge\frac{t}{2b^{\leftarrow}(|y|)}, b(t) \le |y| \le b(\overline{t})=:C'.$$
We can conclude that for any $0 < C < C',$
$$\frac{t}{b(t)}\int_{b(t) < |y| \le 1} |y| \Pi(dy) \le 2\int_{0 < |y| \le C} b^{\leftarrow}(|y|) \Pi(dy) + \frac{t}{b(t)} \int_{|y| > C} |y|\Pi(dy).$$ 
As (\ref{inv-function}) holds,
we can choose for $\epsilon >0$ a positive constant $C=C_{\epsilon} < C' $ so that the first integral will become less than $\epsilon$. \\On the other hand, the second integral is finite for any fixed $C >0$  (as we have $\int_{0  < |y| \le 1} |y|^2 \Pi(dy) < \infty$). Recalling that $t/b(t) \to 0$ as $t\to 0,$ we see that
$$\limsup_{t \to 0} \frac{t}{b(t)}  \int_{b(t) < |y| \le 1} |y| \Pi(dy) \le 2 \epsilon,$$
and part (c) of the lemma has been proven. \end{proof}
We next need a $d$-dimensional version of Lemma 4.3.(i) in \cite{Ber} which we prove for general L\'evy processes.
\begin{lem} \label{moment_3}
 Let $Y_t=(Y_{t,1},\ldots,Y_{t,d}), t \ge 0$ be a $d$-dimensional L\'evy process  with characteristic triplet $(\gamma, \Sigma, \Pi)$ and suppose that its L\'evy measure $\Pi$ has support $D$. Then all moments of $Y_t$ exist and we have $$\E|Y_t|^3/t \to \int |y|^3 \Pi(dy) \mbox{ as }t \to 0.$$
 
\end{lem}
\begin{proof} The existence of the moments follows, for instance, from Theorem 25.3 in \cite{Sato}.
To prove the other assertion of the lemma we first show  for any $x > 0$ satisfying $\Pi\{y: |y|=x\}=0,$ 
\be \label{vague}
\frac{1}{t}\PP\{|Y_t|  > x\} \to \overline{\Pi}(x) \mbox{ as } t \to 0.
\ee
We apply Corollary 8.9. in \cite{Sato}. Set for $x >0$ and $0< \epsilon <x$,
$$f_{x,\epsilon}= 1 \wedge \frac{\mathrm{dist}(\cdot, (x-\epsilon)D)}{\epsilon},$$
where, as usual, $\mathrm{dist}(y, A)=\inf\{|y-z|: z \in A\}$ is the distance of $y$ to the set $A \subset \R^d$. \\
It is easy to see that $f_{x, \epsilon}$ is continuous on $\R^d$  and  we have,
$$I_{(x-\epsilon)D} \le 1 - f_{x, \epsilon} \le I_{xD}.$$
The conclusion is that 
$$\limsup_{t \to 0}\frac{1}{t}\PP\{|Y_t| > x \}\le \limsup_{t \to 0} \frac{1}{t}\E  f_{x,\epsilon}(Y_t) =\int f_{x, \epsilon} d\Pi \le \overline{\Pi}(x-\epsilon).$$
Letting $\epsilon$ converge to zero, it follows that
$$\limsup_{t \to 0}\frac{1}{t}\PP\{|Y_t| > x \}\le \Pi\{y: |y| \ge x\}= \overline{\Pi}(x).$$
A similar argument gives that
$$\liminf_{t \to 0}\frac{1}{t}\PP\{|Y_t| > x \}\ge \overline{\Pi}(x)$$
and relation $(\ref{vague})$ has been proven.\\
After some calculation one obtains from Theorem 25.17 in \cite{Sato} for $1 \le i \le d,$
$$\E Y_{t,i}^2 = t (m_{2,i} + \Sigma_{i,i}) +t^2 \gamma_i^2$$
and
$$\E Y_{t,i}^4 =  m_{4,i}t + (3\{m_{2,i} + \Sigma_{i,i}\}^2 + 4 m_{3,i}\gamma_i)t^2+ 6(m_{2,i} +\Sigma_{i,i})\gamma_i^2 t^3 + \gamma_i^4 t^4 =: \sum_{j=1}^4 c_{j,i}t^j,$$
where  $m_{k,i}=\int y_i ^k \Pi(dy), k \ge 2.$
Arguing as on page 175 of \cite{Ber} we conclude that
\begin{eqnarray*}
\frac{1}{t}\PP\{|Y_t| > x \}&\le& \frac{1}{t}\sum_{i=1}^d \PP\{|Y_{t,i}| \ge x/\sqrt{d}\}\\
&\le&d x^{-2}\sum_{i=1}^d (m_{2,i}+\Sigma_{i,i}+t\gamma_i^2) I_{]0,1]}(x)\\
&&\hspace{2cm} + d^2 x^{-4}\sum_{i=1}^d  (m_{4,i} +\sum_{j=2}^4 c_{j,i}t^{j-1}) I_{]1,\infty[}(x),
\end{eqnarray*}
As there are at most countably many $x> 0$ for which $\Pi\{y: |y| =x\} > 0$, we have convergence almost everywhere in (\ref{vague}) so that we can apply the  dominated convergence theorem. We see  that as $t \to 0,$
$$\frac{1}{t} \E |Y_t|^3 = \int_0^{\infty} 3x^2\frac{1}{t}\PP\{|Y_t| > x \}dx \to \int |y|^3 \Pi(dy).$$
The lemma has been proven. \end{proof}
We return to the special case where the L\'evy  process $\{X_t: t \ge 0\}$ is purely non-Gaussian.  As in \cite{Ber, Sav}   we  then can write the stochastic process $X_t$ as a sum of two (independent) stochastic processes $Y_t^{(b)}$ and $Z_t^{(b)}$ plus a deterministic term $\nu(b)$. 
Letting $\Delta X_t = X_t - X_{t-}, t > 0$  be the jumps of $X_t,$ it follows from the L\'evy-It\^o decomposition that  for any $0 < b \le 1$,
\be \label{Ito}
X_t = t\nu(b) + Y_t^{(b)} + Z_t^{(b)}, t >0,
\ee
where
$$\nu(b)= \gamma - \int_{b < |y| \le 1} y \Pi(dy),$$
$$Y_t^{(b)}= \mbox{a.s. } \lim_{\epsilon \downarrow 0} \left\{ \sum_{0 < s \le t} \Delta X_s I_{\{\epsilon <|\Delta X_s| \le b\}} - t \int_{\epsilon < |y| \le b} y \Pi(dy)\right\}$$
and
$$Z_t^{(b)}=\sum_{0 < s \le t} \Delta X_s I_{\{|\Delta X_s| > b\}}.$$
As in  Lemma 4.1 of \cite{Sav} we see that under condition (\ref{3}) one has for any $0 < r <1,$
\be \label{trunc}
 \PP\left\{\sup_{0 \le s \le r^n} |Z_s^{(b(r^n))}|>0 \mbox{ i.o.}\right\}= 0.
 \ee
This will enable us to reduce the proof of Theorem \ref{dimension_d} to studying the processes
 $$Y_t^{(b(r^n))}, r^{n+1} < t \le r^n, n \ge 0.$$
We first look at the upper bound in Theorem \ref{dimension_d}. 
\subsection{The upper bound part}
Using  that  for any sequence $\{\xi_n: n \ge 1\}$ of i.i.d. mean zero random vectors with $\E |\xi_1|^2 < \infty$,
$$\E|\sum_{i=1}^n \xi_i| \le (\E |\sum_{i=1}^n \xi_i|^2)^{1/2}=\sqrt{n}(\E |\xi_1|^2)^{1/2}$$
we  can infer from  Theorem 3.1  in \cite{EL},
\begin{theo} Let $\xi_1,\ldots, \xi_n$ be i.i.d.  mean zero random vectors with finite third  moments. Then we have for any fixed $\delta >0$ and all $x >0,$
$$\PP\left\{\max_{1 \le k \le n} |\sum_{i=1}^n \xi_i| \ge 2\sqrt{n}(\E|\xi_1|^2)^{1/2} + x\right\}
\le \exp\left(- \frac{x^2}{(2+\delta)n\Lambda}\right) + C n \E |\xi_1|^3 /x^3,$$
where $\Lambda=\sup\left\{\E \langle z, \xi_1\rangle^2 : |z| \le 1 \right\}$ and $C$ is a positive constant depending on $\delta$.
 \end{theo}
This implies the following inequality for the L\'evy process $\{Y_s^{(b)}, s \ge 0\},$ 
\begin{theo}\label{FNI}
 Given $\delta >0$ we have for all $t, x >0,$
\begin{eqnarray*}
&&\PP\left\{\sup_{0 \le s\le t} | Y_s^{(b)}| \ge 2(\E|Y_t^{(b)}|^2)^{1/2} + x\right\}\\
&\le& \exp\left(- \frac{x^2}{(2+\delta)tV(b)}\right) + C t \int_{0 <|y| \le b} |y|^3 \Pi(dy) /x^3,
\end{eqnarray*}
where $C >0$ is a constant depending on $\delta.$
 \end{theo}
 \begin{proof}  By the right continuity of the sample paths $s \mapsto Y_s^{(b)},$ the above probability is
 equal to
 $$
 \lim_{n \to \infty} \PP\left\{\max_{1 \le k\le n} | Y_{tk/n}^{(b)}| \ge 2(\E|Y_t^{(b)}|^2)^{1/2} + x\right\}=:\lim_{n \to \infty} p_n.$$
 Note that $$Y_{tk/n}^{(b)}=\sum_{j=1}^k (Y_{tj/n}^{(b)}-Y_{t(j-1)/n}^{(b)}), 1 \le k \le n,$$
 where the random vectors $\xi_{j,n,b}:=Y_{tj/n}^{(b)}-Y_{t(j-1)/n}^{(b)}, 1 \le j \le n$ are i.i.d.\\
 Moreover we have
$$\E \xi_{1,n,b}=\E Y_{t/n}^{(b)}=0 \mbox{ and }n \E |\xi_{1,n,b}|^2 = \E{ |Y^{(b)}_t|}^2,$$
By Example 25.12 in \cite{Sato} we also know that
$$\E\langle z, \xi_{1,n,b}\rangle^2 = (t/n)\int _{|y| \le b}\langle z, y\rangle^2\Pi(dy), z \in \R^d.$$
We can infer that $\Lambda$ in the above inequality is equal to $ V(b) t/n.$\\
It follows that
$$p_n
\le \exp\left(- \frac{x^2}{(2+\delta)tV(b)}\right) + Ct \frac{n}{t} \E | Y_{t/n}^{(b)}|^3 /x^3.$$
Letting $n$ go to infinity and recalling Lemma \ref{moment_3}, the inequality follows. \end{proof}
Note also that 
\be \label{sec-mom}
\E|Y_t^{(b)}|^2 \le d tV(b), t \ge 0, 0 < b  \le 1.
\ee
We are ready to establish the upper bound in Theorem \ref{dimension_d}, that is, 
\be \label{upper1}
\limsup_{t \downarrow 0} \frac{|X_t|}{b(t)} \le \alpha_0 \mbox{ with prob. 1}
\ee
where w.l.o.g we can and do assume that $\alpha_0 < \infty$. \\
By definition of $\alpha_0$ we have for any $\alpha > \alpha_0,$
\begin{eqnarray}
\infty &>& \sum_{n=0}^{\infty}\int_{r^{n+1}}^{r^n} t^{-1}\exp\left(-\frac{\alpha^2 b^2(t)}{2t V(b(t))}\right)dt  \label{series1}\\
&\ge&\log(r^{-1})\sum_{n=0}^{\infty}  \exp\left(-\frac{\alpha^2 b^2(r^n)}{2r^{n+1}V(b(r^{n+1}))}\right).\nonumber
\end{eqnarray}
In particular, we have $r^n V(b(r^n))/ b^2(r^{n-1}) \to 0$ as $n \to \infty$.\\ From this observation it further follows (see \ref{sec-mom}) that
\be
c_n:=\E \left|Y_{r^n}^{(b(r^n))}\right|^2 = o(b^2(r^{n-1})) \mbox{ as }n \to \infty.
\ee
Let $\delta >0$. Then it follows from Theorem \ref{FNI} that for large $n$ 
\begin{eqnarray*}
&&\PP\left\{\sup_{0 \le s \le r^n} \left|Y_{s}^{(b(r^n))}\right| \ge (1+ \delta)( \alpha_0 + \delta)  b(r^{n-1}) \right\}\\
&\le& \PP\left\{\sup_{0 \le s \le r^n} \left|Y_{s}^{(b(r^n))}\right| \ge (\alpha_0+\delta)(1+\delta/2) b(r^{n-1}) + 2c_n^{1/2} \right\}\\
&\le&\exp\left(-\frac{(\alpha_0+\delta)^2 b^2(r^{n-1})}{2r^n V(b(r^n))}\right) + C' r^n \int_{0 < |y| \le b(r^{n})} |y|^3 \Pi(dy)/b^3(r^{n-1}),
\end{eqnarray*}
where $C' >0$ is a constant depending on $\delta.$\\
From relation (\ref{series1}) it follows that
$$\sum_{n=1}^{\infty} \exp\left(-\frac{(\alpha_0+\delta)^2 b^2(r^{n-1})}{2r^n V(b(r^n))}\right) < \infty.$$
Moreover employing the same argument already used for proving (\ref{series1}), we can infer from Lemma \ref{cond}(a) that
$$\sum_{n=1}^{\infty} r^n \int_{0 < |y| \le b(r^{n})} |y|^3 \Pi(dy)/b^3(r^{n-1}) < \infty.$$
By the Borel-Cantelli lemma we then have for any $\delta >0$  with probability one,
$$\sup_{r^{n+1} < s \le r^n} \left|Y_{s}^{(b(r^n))}\right| \le (1+ \delta)( \alpha_0 + \delta)  b(r^{n-1}), \mbox{ eventually}$$ 
Combining this with Lemma \ref{cond}(c)  and relation (\ref{trunc}), we see that  with probability one,
$$\limsup_{t \to 0} \frac{|X_t|}{b(t/r^2)} \le \alpha_0.$$
From condition (\ref{50})   it  follows that   $\limsup_{t \to 0}b(t/r^2)/b(t) \le r^{-2}$ for any fixed $0 < r < 1.$ Thus, we have with with probability one,
$$\limsup_{t \to 0} \frac{|X_t|}{b(t)} \le \alpha_0 r^{-2}.$$
Since this holds for any $r < 1$, relation (\ref{upper1}) follows.

\subsection{The lower bound part}
W.l.o.g. we assume $\alpha_0 > 0.$ We show that  for any $ 0 < \alpha < \alpha_0$ with probability 1,
\be \label{lower_bound}
\limsup_{t \to 0} |X_t| / b(t)\ge \alpha.
\ee
If $\int_0^1 \overline{\Pi} (b(t)) dt = \infty,$ it follows from Lemma \ref{cond}(b) that we have 
$$\int_0^1 \overline{\Pi} (Cb(t)) dt = \infty,\forall\, C > 0$$
and we can infer from Proposition 4.2 in \cite{Ber} that with probability one $\limsup_{t \to 0} |X_t| / b(t)=\infty$ and (\ref{lower_bound}) is trivially true.\\[.1cm]
This is also the case if $\limsup_{t \to 0} \PP\{|X_t| \ge \alpha b(t)\} >0.$ For this implies that
$$\PP\{\limsup_{t \to 0} |X_t|/b(t) \ge \alpha\} \ge \limsup_{t \to 0} \PP\{|X_t| \ge \alpha b(t)\} >0.$$
Then the first probability has to be equal to 1 by  Blumenthal's 0-1 law (see, \cite{Sato}, Proposition 40.4) and (\ref{lower_bound}) holds.\\[.1cm]
So it is enough to prove (\ref{lower_bound}) under the assumptions (\ref{3}) and
\be  \label{59}
\PP\{|X_t| \ge \alpha b(t)\}  \to 0 \mbox{ as } t \to 0.
\ee
To that end we first show for $\alpha < \alpha_0,$
\be \label{diverg}
\sum_{n=0}^{\infty} \PP\{|X_{r^n}| \ge \alpha b(r^n)\} = \infty,
\ee
where $0 < r < 1$ has to be chosen so that $\alpha/r < \alpha_0.$\\
Using the same argument as in the proof of relation (\ref{series1}), we find that for any $0< \tilde{\alpha} < \alpha_0,$
$$\sum_{n=0}^{\infty} \exp\left(-\frac{\tilde{\alpha}^2 b^2(r^{n+1})}{2r^{n}V(b(r^{n}))}\right)=\infty.$$
As we have by condition (\ref{50})  for $0 < \epsilon <1,$ $b(r^{n+1})\ge r b(r^n) (1-\epsilon)$ if $n$ is large, we can conclude that for $0 < r_1 <r,$
\be \label{series2}
\sum_{n=0}^{\infty} \exp\left(-\frac{\tilde{\alpha}^2 r_1^2 b^2(r^{n})}{2r^{n}V(b(r^n))}\right)=\infty.
\ee
Let $\alpha < \alpha_1 < \alpha_0$ be chosen so that we still  have $\alpha_1/r < \alpha_0.$ Then there exist an $r_1$ satisfying $0 < r_1 < r$  and $\delta > 0$ small so that $\tilde{\alpha}:=\alpha_1(1+\delta)/r_1 < \alpha_0.$ 
It follows  from (\ref{series2}) that
\be \label{series3}
\sum_{n=0}^{\infty} \exp\left(-\frac{\alpha_1^2 (1 +\delta)^2  b^2(r^{n})}{2r^{n}V(b(r^n))}\right)=\infty.
\ee
Next observe that
\begin{equation*}
\begin{split}
 \PP\{|X_{r^n}| \ge \alpha b(r^n)\} \ge &\; \PP\{|Y_{r^n}^{(b(r^n))}| \ge \alpha_1 b(r^n)\}\\
 &\quad - \PP\{|Z_{r^n}^{(b(r^n))} + r^n \nu(b(r^n))| \ge (\alpha_1-\alpha) b(r^n)\}.
\end{split}
\end{equation*}
From Lemma \ref{cond}(c) in combination with relation (\ref{trunc}), we readily obtain that
$$\sum_{n=0}^{\infty} \PP\{|Z_{r^n}^{(b(r^n))} + r^n \nu(b(r^n))| \ge (\alpha_1-\alpha) b(r^n)\} < \infty.$$
Therefore, (\ref{diverg}) is proven once we have shown that
\be \label{diverg1}
\sum_{n=0}^{\infty} \PP\{|Y_{r^n}^{(b(r^n))}| \ge \alpha_1 b(r^n)\}=\infty.
\ee
To establish this relation, we first derive an inequality which gives lower bounds for the probabilities $\PP\{|Y_t^{(b)}| \ge x\}, x > 0.$
\begin{lem} \label{lowerboundsum}
 Let $ \xi_1, \ldots, \xi_n$ be i.i.d. mean zero random vectors with finite third  moments.
 Then we have for any $0 < \delta <1$ and all $x >0,$
 $$\PP\{|\sum_{j=1}^n \xi_j| \ge x\} \ge C_1\exp\left(-\frac{x^2(1+\delta)^2}{2n\Lambda}\right) - C_2n \E\ |\xi_1|^3/x^3,$$
 where $\Lambda =\sup\{\E \langle \xi_1, z \rangle^2 : |z| \le 1\}$ and $C_i, i=1,2 $ are positive constants depending on $\delta$ only.
\end{lem}
 \begin{proof} Applying Lemma 5 in \cite{E-3} with $s=x$ and $t=x\delta/2$, we find that
 $$
 \PP\{|\sum_{j=1}^n \xi_j| \ge x\} \ge \PP\{|\sum_{j=1}^n \eta_j| \ge x(1+\delta/2)\} -
 8A\delta^{-3} n \E |X_1|^3 x^{-3},
 $$
 where the random vectors  $\eta_1,\ldots,\eta_n$ are i.i.d.  with $\mathcal{N}(0, \mathrm{cov}(\xi_1))$-distribution, $\mathrm{cov}(\xi_1)$ is the covariance matrix of $\xi_1$ and $A$ is an absolute constant.\\
 Next choose a unit vector $z \in \R^d$ so that $\Lambda = \E \langle \xi_1, z \rangle^2 = \E \langle \eta_1, z \rangle^2$.\\ Then we have
 \begin{eqnarray*}
\PP\{|\sum_{j=1}^n \eta_j| \ge x(1+\delta/2)\}&\ge &\PP\{|\sum_{j=1}^n \langle \eta_j, z \rangle| \ge x(1+\delta/2)\}\\
&=&\PP\{\sqrt{n} |\eta'| \ge x(1+\delta/2)/\Lambda^{1/2}\}
\end{eqnarray*}
 where $\eta'$ is  a 1-dimensional standard normal random variable. Using the inequality 
 \be \label{normal_lower}
\PP\{\eta' \ge x\} \ge x (x^2 +1)^{-1}\exp(-x^2/2)/\sqrt{2\pi}, x > 0,
 \ee
  we easily obtain the above lower bound. \end{proof}
 By the same reasoning as in the proof of Theorem \ref{FNI} we obtain the following result for the L\'evy processes $\{Y_t^{(b)}: t \ge 0\}$ from Lemma \ref{lowerboundsum},
 \begin{theo}\label{FNI-reversed}
 Given $\delta >0$ we have for all $t, x >0,$
$$\PP\left\{ | Y_t^{(b)}| \ge  x\right\}
\ge C_1\exp\left(- \frac{x^2(1+\delta)^2}{2tV(b)}\right) -C_2 t \int_{0 <|y| \le b} |y|^3 \Pi(dy) /x^3,$$
where $C_1, C_2 >0$ are constant depending on $\delta$ only.
\end{theo}
Recalling Lemma \ref{cond}(a) and (\ref{series3}), we now see that
\begin{eqnarray*}
\sum_{n=0}^{\infty} \PP\{|Y_{r^n}^{(b(r^n))}| \ge \alpha_1 b(r^n)\} 
&\ge& C_1\sum_{n=0}^{\infty} \exp\left(-\frac{\alpha_1^2 (1 +\delta)^2  b^2(r^{n})}{2r^{n}V(b(r^n))}
\right)\\
 &-& C_2\alpha_1^{-3} \sum_{n=0}^{\infty} r^n \int_{ |y| \le b(r^n)} |y|^3 \Pi(dy)/b^3(r^n)=\infty.
\end{eqnarray*}
This shows that (\ref{diverg1}) holds and we thus have proven (\ref{diverg}).\\[.1cm]
As in \cite{Sav} (see formula (4.7)) we can infer from (\ref{diverg}) that for any fixed natural number $m \ge 1,$ there exists a $k \in \{0,\ldots, m-1\}$ such that
\be \label{diverg5}
\sum_{n=0}^{\infty} \PP\{|X_{r^{nm+k}}| \ge \alpha b(r^{nm+k})\} =\infty.
\ee
Recall that $0 < r <1$ had to be chosen so that $\alpha/r < \alpha_0$. We will assume that $r^m < 1/2$ which holds if $m$ is bigger than some finite positive number $m_0=m_0(r).$\\
Further set $t_n= r^{nm+k}/(1-r^m), n \ge 0$ which implies that $t_n-t_{n+1}=r^{nm+k}$ and also
$t_{n+1} \le (1-r^m)t_n.$\\ Let $0 <\delta  <1$ be fixed. Consider for $n \ge 1$ the events
$$A_n:= \{|X_{t_n} - X_{t_{n+1}}| \ge \alpha b(t_n(1-r^m))\}, B_n:=\{|X_{t_{n+1}}| \le \delta b(t_n(1-r^m))\}.$$
Then we have by condition (\ref{49}), $b(t_n(1-r^m))/b(t_{n+1}) \ge (r^{-m}/2)^{-\rho}$ for large $n$. Choosing $m \ge m_1$ for a suitable number $m_1=m_1(r, \alpha, \delta, \rho) \ge m_0$ , it follows that for large $n,$
$$\PP(B_n^c) \le \PP\{|X_{t_{n+1}}| \ge \alpha b(t_{n+1})\},$$
which converges to zero by  (\ref{59}).\\
The events $A_n$ are independent and we have by (\ref{diverg5})
$$\sum_{n=1}^{\infty} \PP(A_n) = \infty.$$
Moreover, $A_1,\ldots, A_n, B_n$ are independent for any $n \ge 1$. Thus, we can infer from the Feller-Chung lemma (see Lemma 3(i) on page 70 in \cite{CT}) that
$$\PP(A_n \cap B_n \mbox{ infinitely often}) = \PP(A_n  \mbox{ infinitely often}) =1.$$
Finally, note that
$$\{A_n \cap B_n \mbox{ infinitely often}\} \subset \{\limsup_{t \to 0}|X_t|/b(t(1-r^m)) \ge \alpha - \delta\}.$$
In view of condition (\ref{50}), we have $\liminf_{t \to 0} b(t(1-r^m))/b(t) \ge 1 -r^m$ which is $ \ge (1-\delta)$ if $m$ is big enough.
We now can conclude that with probability one,
$$\limsup_{t \to 0}|X_t|/b(t) \ge (1-\delta)(\alpha - \delta).$$
As $\delta$ can be chosen arbitrarily small, we see that (\ref{lower_bound}) holds. Theorem \ref{dimension_d} has been proven. \qed\\[.2cm]
To conclude this section we show that Theorem \ref{dimension_d} remains true if we replace the function $V$ by $V_1$. 
To prove this it is enough to show that $\alpha_1=\alpha_0,$ where
\be  \label{V_1}
\alpha_1: = \sup\left\{\alpha \ge 0: \int_0^1 \frac{1}{t} \;\exp\left(-\frac{\alpha^2 b^2(t)}{2t V_1(b(t))}\right) dt = \infty\right\}.
\ee
It is obvious  that $\alpha_0 \le \alpha_1$ (since $V(b(t)) \le V_1(b(t))$).\\
For the reverse inequality we can assume w.l.o.g. that $\alpha_0 < \infty.$\\
Observe that we have 
for any  $t>0,$
\begin{eqnarray*}
V_1(t)=\sup_{|z| =1}\int_{|\langle y, z\rangle| \le t} \langle y, z\rangle^2 \Pi(dy) &\le& 
V(t) + \sup_{|z| \le 1} \int_{\{|y| > t, |\langle y, z\rangle| \le t\}} \langle y, z\rangle^2 \Pi(dy)\\
&\le&V(t)+t^2\overline{\Pi}(t).
\end{eqnarray*}
Given $0 < \delta <1$, set  $\alpha = (1+\delta)( \alpha_0 + \delta)$. Then we have 
\begin{eqnarray*}
&&\int_0^1 \frac{1}{t} \;\exp\left(-\frac{\alpha^2 b^2(t)}{2t V_1(b(t))}\right) dt \\
&\le& \int_0^1 \frac{1}{t} \;\exp\left(-\frac{\alpha^2 b^2(t)}{2t (V(b(t))+ b^2(t)\overline{\Pi}(b(t)))}\right) dt \\
&\le& \int_0^1 \frac{1}{t} \;\exp\left(-\frac{\alpha^2 b^2(t)}{2(1+\delta)t V(b(t))}\right) dt\\ 
&&\hspace{2cm} + \int_0^1 \frac{1}{t} \;\exp\left(-\frac{\alpha^2 b^2(t)}{2t(1+\delta^{-1})b^2(t)\overline{\Pi}(b(t))) }\right) dt \\
&\le& \int_0^1 \frac{1}{t} \;\exp\left(-\frac{(\alpha_0+\delta)^2  b^2(t)}{2t V(b(t))}\right) dt 
+ \frac{2(1+\delta^{-1})}{\alpha^2}\int_0^1 \overline{\Pi}(b(t))dt < \infty,
\end{eqnarray*}
where we have used in the last step the  inequality $\exp(-x) \le x^{-1}, x > 0.$\\
We see that $\alpha_1 \le (\alpha_0 + \delta) (1 + \delta), \delta > 0$ and thus $\alpha_0 \ge \alpha_1.$

 \section{Proofs of Theorems \ref{upper}, \ref{lower} and Lemma \ref{extra-lemma}}
 To simplify our notation, we set for any $\alpha \ge 0,$
 $$J(\alpha) := \int_0^1 \frac{1}{t} \;\exp\left(-\frac{\alpha^2 b^2(t)}{2t V(b(t))}\right) dt.$$
 Note that $J(\alpha)$ is finite if there exists a $0 < u_0 <1$ such that the integral over $[0,u_0]$ is finite. \\[.2cm]
{\it Proof of Theorem \ref{upper}.}\\
W.l.o.g. we assume that $\lambda < \infty.$ In view of Theorem \ref{dimension_d} it is enough to show that $J(\lambda + \delta) < \infty, \delta > 0.$ \\Choosing $0 < t_{\delta} < e^{-1}$ small enough so that  
$$b(t) = \sqrt{t \log \log 1/t}/h(1/t), 0 < t \le t_{\delta}$$ 
and 
$$V(b(t)) \le \frac{(\lambda + \delta/2)^2}{2} h^{-2}(1/t), 0 < t < t_{\delta},$$
we readily obtain that
$$\int_0^{t_{\delta}} \exp(-(\lambda +\delta)^2 b^2(t)/(2t V(b(t)))t^{-1}dt
\le \int_0^{t_{\delta}} t^{-1} (\log 1/t )^{-(1+\epsilon)^2} dt $$
where $\epsilon = (\lambda+\delta)/ (\lambda+\delta/2) - 1 >0.$
It  follows that $J(\lambda+\delta)<\infty.$ \qed\\[.2cm]
{\it Proof of Theorem \ref{lower}.}\\
It is obviously enough to prove this result if $\lambda > 0.$ Moreover,  as in the lower bound part proof of Theorem \ref{dimension_d},  we can assume w.l.o.g. that
$$\int_0^1 \overline{\Pi} (b(t))dt < \infty$$
since otherwise the $\limsup$ is infinite and Theorem \ref{lower} is certainly true.\\
Then Theorem \ref{dimension_d} applies and it it is enough to show that
$$J(\alpha)=\infty, 0 < \alpha < (1-q)^{1/2}\lambda.$$
Given $0 < \alpha < (1-q)^{1/2}\lambda,$ let $\tau'$ be defined by $\alpha^2=\tau' \lambda^2$ and  take a $\tau$ satisfying 
$$0 < \tau' < \tau < (1-q).$$
Next choose a sequence $e^{-1}  > t_k \downarrow 0$  such that
$$V(b(t_k))\ge \frac{\lambda^2}{2}h^{-2}(1/t_k)(1-1/k), k \ge 1$$
and set
$$\tilde{t}_k = t_k\exp((\log 1/t_k)^\tau /2), k \ge 1.$$
A small calculation gives that 
$$\tilde{t}_k^{-1}f_{\tau}({\tilde{t}}_k^{-1}) \ge t_k^{-1}, k \ge 1.$$
Using the fact that
$$h(\tilde{t}_k^{-1}f_{\tau}(\tilde{t}_k^{-1}))/h(\tilde{t}_k^{-1}) \to 1 \mbox{ as }k \to \infty,$$
we can conclude that
$$h(1/t_k)/h(1/\tilde{t}_k) \to 1 \mbox{ as }k \to \infty.$$
By monotonicity of the functions $V$ and $b$ we obtain  for some $\delta_k \to 0$,
$$V(b(t)) \ge  \frac{\lambda^2}{2}h^{-2}(1/t)(1-\delta_k), t_k \le t \le \tilde{t}_k,$$
Consequently, we have
$$
J(\alpha) \ge \int_{t_k}^{\tilde{t}_k} t^{-1} \log(1/t)^{-\tau'/(1-\delta_k)} dt\ge  \log(\tilde{t}_k/t_k) (\log 1/t_k)^{-\tau'/(1-\delta_k)}
$$
which is by definition of $\tilde{t}_k,$
$$\ge  \log(1/t_k)^{\tau - \tau'/(1-\delta_k)}/2.$$
This sequence converges to infinity and it is now clear that $J(\alpha)=\infty,$ whenever $0 < \alpha < (1-q)^{1/2}.$ Theorem \ref{lower} has been proven. \qed\\[.2cm]
To prove Lemma \ref{extra-lemma} we need a further lemma.
\begin{lem} Assume that $\overline{V}(t):= \int_{|y| \le t} |y|^2 \Pi(dy) >0, t > 0.$ Let $g: [0,\infty[ \to ]0, \infty[$ be a  non-decreasing function such that  $\int_0^c g(t)^{-1} dt < \infty$ for all $ c >0.$ 
Then we have,
$$\int_{|y| \le 1} \frac{|y|^2}{g(\overline{V}(|y|))}\Pi(dy) < \infty.$$
 \end{lem}
 \begin{proof} We first  note that
 $$ \int_{|y| \le 1} \frac{|y|^2}{g(\overline{V}(|y|))}\Pi(dy)=\int_0^1  \frac{t^2}{g(\overline{V}(t))}Q(dt),
 $$
where $Q$ is the image measure $\Pi_f$ with $f:\R^d \to \R$ being the Euclidean norm.\\
Set $\overline{V}(1)=:K.$  Let further $R$ be the p-measure on the Borel subsets of $\R$ with $Q$-density  $K^{-1}t^2 I_{[0,1]}(t)$ and note that $F(t):=\overline{V}(t)/K$ is the distribution function of $R$. We then have
$$\int_0^1  \frac{t^2}{g(\overline{V}(t))}Q(dt) = K \int_0^1  \frac{1}{g(K F(t))} R(dt).$$
 Consider the generalized inverse function of $F$, that is
 $$\phi(u)=\inf\{t \ge 0: F(t) \ge u\}, 0 < u < 1.$$
 Then it is easy to see that $F(\phi(u)) \ge u, 0 < u < 1.$ Moreover, if $U: \Omega \to ]0,1[$ is a uniform$(0,1)$-distributed random variable on a p-space $(\Omega,\fc,\PP)$, the random variable $\phi \circ U$ has distribution function $F$.
  It follows that
  \begin{eqnarray*}
\int_0^1 g(K F(t))^{-1} R(dt) &=& \int g(KF(\phi(U)))^{-1}d\PP\\
&=&\int_0^1 g(KF(\phi(u)))^{-1}du \le \int_0^1 g(Ku)^{-1} du  < \infty
\end{eqnarray*}
 and the lemma has been proven. \end{proof}
 Recalling that $V(t) \le \overline{V}(t) \le d V(t)$, we can infer that if $V(t) > 0, t >0,$ we have for any $\delta > 0,$
\be \label{41}
 \int_{|y| \le 1} \frac{|y|^2}{V(|y|)(\log_{+}(1/V(|y|)))^{1+ \delta}}\Pi(dy) < \infty.
 \ee
 Here we set $\log_{+}(x) =1 \vee \log x, x > 0.$\\[.2cm]
\emph{Proof of Lemma \ref{extra-lemma}}\\
W.l.o.g. we assume that $V(t)> 0, t >0$ and that $h(x) \to \infty$ as $x\to \infty.$ Otherwise the lemma is trivial.\\
The function $t \mapsto \sqrt{t \log \log 1/t}/h(1/t)=b(t)$ is increasing and invertible on a subinterval $[0,t_0]$. We then clearly have for $0 \le x \le b(t_0),$
$$b^{\leftarrow}(x) =  x^2h^2(1/b^{\leftarrow}(x))/\log \log (1/b^{\leftarrow}(x)) .$$
Since  $V(b(t))= O( h^{-2}(1/t))$ as $t \to 0$ it follows that 
$$V(x) \le C_1 h^{-2}(1/b^{\leftarrow}(x)), 0 < x \le b(t_0),$$ where $C_1 \ge 1$ is a 
 positive constant. Combining this inequality with (\ref{xy}), we find that for small enough $x,$
\begin{eqnarray*}
V(x)(\log_{+} 1/V(x))^{1/\vartheta} &\le& C_1 h^{-2}(1/b^{\leftarrow}(x))(\log( h^2(1/b^{\leftarrow}(x)))^{1/\vartheta}\\
&\le& C_1 h^{-2}(1/b^{\leftarrow}(x)) ( 2(\log \log 1/b^{\leftarrow}(x))^{\vartheta})^{1/\vartheta}\\
&\le& C_2 h^{-2}(1/b^{\leftarrow}(x)) \log \log 1/b^{\leftarrow}(x),
\end{eqnarray*}
where $C_2$ is a positive constant.\
Thus we have for some $a_0 \le e^{-2},$
$$\int_{0 < |y| \le a_0} b^{\leftarrow}(|y|)\Pi(dy) \le  C_2\int_{0 <|y| \le a_0} \frac{|y|^2}{V(|y|)(\log_{+}(1/V(|y|)))^{1/\vartheta}}\Pi(dy),$$
where the second integral is finite by (\ref{41}). (Recall that $\vartheta < 1.$)\\[.2cm]
This implies that $\int_{0 < |y| \le 1}  b^{\leftarrow}(|y|) \Pi(dy) < \infty$ and our proof of Lemma \ref{extra-lemma} is complete. \qed

\section{Cluster sets}
\subsection{A general result}
Throughout this subsection $X_t, t\ge 0$ will be a (general) $d$-dimensional L\'evy process such that $X_t/b(t)$ is stochastically bounded as $t \to 0,$ that is,
\be \label{stoch_bounded}
 \forall \delta >0 \,\exists K_{\delta}> 0, 0 < t_{\delta} < 1:\, 
\PP\{|X_t| > K_{\delta} b(t)\} < \delta, 0 < t < t_{\delta}.
\ee
As in the previous sections $b: [0,1]\to [0,1]$ is a continuous and increasing function such that $b(1)=1$,  $b(t) /t \to \infty$ as $t \to 0$ and conditions (\ref{49}), (\ref{50}) are satisfied.
\begin{lem} \label{Kesten1} Let $0 < r < 1$ be a fixed number and let $x \in \R^d.$ The following are equivalent:
\begin{itemize}
\item[(a)] $x \in C(\{X_t/b(t): t \downarrow 0\})$ with probability one
\item[(b)] $ \sum_{n=1}^{\infty} \PP\{|X_t/b(t) -x| < \epsilon \mbox{ for some } r^{n+1} \le t < r^n\} =\infty$ for all $\epsilon > 0.$
\end{itemize}
\end{lem}
\begin{proof} \fbox{(a) $\Rightarrow$ (b)} This follows directly from the Borel-Cantelli lemma.\\[.1cm]
\fbox{(b) $\Rightarrow$ (a)} It is obviously enough to show that we have for any $\epsilon >0,$
$$\PP(\limsup_{n \ge 1} \{|X_t/b(t) - x| < \epsilon \mbox{ for some } t_{n+1} \le t < t_n\})=1,$$
where we set $ t_n = ar^n$ for some constant $a  >0$ which will be specified later on.
This is equivalent to proving
\be \label{uni}
\PP(\bigcup_{n=N}^{\infty} \{|X_t/b(t) - x| < \epsilon \mbox{ for some } t_{n+1} \le t < t_n\})=1, \forall N \ge 1.
\ee
Take an $m \ge 2$ and set
$$A_n:= \{|(X_t -X_{t_{n+m}})/b(t) - x| < \epsilon/2 \mbox{ for some } t_{n+1} \le t < t_n\}, n \ge 1$$
Then the probability of the union in (\ref{uni}) is
$$\ge \PP\left(\bigcup_{n=N}^{\infty} A_n \cap \{|X_{t_{n+m}}| < \epsilon b(t_{n+1})/2\}\right)$$
which is by the Feller-Chung lemma (see Lemma 3(i) on page 70 in \cite{CT})
$$\ge  \PP\left(\bigcup_{n=N}^{\infty} A_n\right) \inf_{n \ge N} \PP \{|X_{t_{n+m}}| < \epsilon b(t_{n+1})/2\}.$$
Observe that 
$$\PP(A_n)=\PP\{|X_{t -t_{n+m}}/b(t) - x| < \epsilon/2 \mbox{ for some } t_{n+1} \le t < t_n\}
$$
which is
\begin{eqnarray*}
&\ge&\PP\left\{\left|\frac{X_{t -t_{n+m}}}{b(t-{t_{n+m}})} - x\frac{b(t)}{b(t-t_{n+m})}\right| < \epsilon/2\mbox{ for some } t_{n+1} \le t < t_n\right\}\\
&\ge&\PP\left\{\left|\frac{X_{s}}{b(s)} - x\,\right| < \frac{\epsilon}{2} -|x|(c_{n,m} -1)
 \mbox{ for some }  t_{n+1}-t_{n+m} \le s < t_n-t_{n+m}\right\},
\end{eqnarray*}
where $c_{n,m}:= \sup_{t_{n+1} \le t \le{t_n}} b(t)/b(t-t_{n+m}).$\\
If $x \ne 0$  we set $\delta =\epsilon |x|^{-1}/8$. By relation (\ref{50}) we have for $t_{n+1} \le t \le t_n$  and large enough $n,$ 
$$b(t)/b(t-t_{n+m}) \le (1+\delta)t/(t-t_{n+m})\le (1+\delta)(1-r^{m-1})^{-1} \le 1+\epsilon|x|^{-1}/4, $$
provided that $m$ is bigger than some  $m_2$ (which depends on $x$ and $\epsilon$). \\ We can conclude that
$$\PP(A_n) \ge \PP\{|X_{s}/b(s) - x| < \epsilon/4 
 \mbox{ for some } t_{n+1}-t_{n+m} \le s < t_n-t_{n+m}\}.$$
 This also holds  if $x=0$ (in which case the last argument is unnecessary).\\
Setting $a=(1-r^{m-1})^{-1}$, we have $t_{n+1}-t_{n+m} =r^{n+1}$ and $t_{n}-t_{n+m} \ge r^n$
and we find that
$$\PP(A_n) \ge \PP\{|X(s)/b(s)-x| < \epsilon/4 \mbox{ for some }r^{n+1} \le s  < r^n\}$$ and consequently we have
$$\sum_{n=1}^{\infty} \PP(A_n) = \infty.$$
The events $A_n$ are defined so that  $A_i$ and $A_j$ are independent if $|i-j| \ge m.$ Choosing a $k \in \{1,\ldots, m-1\}$ such that $\sum_{j=1}^{\infty} \PP(A_{jm+k})=\infty$
 it easily follows from the Borel-Cantelli lemma for pairwise independent events that
$\PP(\limsup_{j \ge 1} A_{jm+k}) =1,$ which of course implies that
$$\PP(\bigcup_{n=N}^{\infty} A_n)=1, \forall\,N \ge1.$$
Finally, given $0 < \delta <1$, we can find for  $K_{\delta} >0$ as defined in (\ref{stoch_bounded})  a natural number $N_{\delta}$ such that 
$$\PP\{|X_{t_n}| \le K_{\delta} b(t_n)\} \ge (1-\delta), \forall n \ge N_{\delta}.$$
In view of (\ref{49}) we have  $b(t_{n+1}) /b(t_{n+m})\ge r^{-\rho (m-1)}$ for some $\rho > 1/3.$ If  $m$  is large enough the last term is $\ge 2K_{\delta}/\epsilon.$ We can conclude that
$$\PP\{|X_{t_{n+m}}| \le \epsilon b(t_{n+1})/2\} \ge 1-\delta, n \ge N_{\delta}.$$
It follows that for $N \ge N_{\delta}$ and then trivially for all $N \ge 1,$
$$
\PP(\bigcup_{n=N}^{\infty} \{|X_t/b(t) - x| < \epsilon \mbox{ for some } t_{n+1} \le t < t_n\}) \ge 1-\delta
$$
and the lemma has been proven. \end{proof}
We are now able to prove the following criterion for clustering of normalized L\'evy processes at zero which is analogous to a well known result of Kesten \cite{Ke} for random walks.
\begin{theo}\label{Kesten2} Let $\{X_t: t \ge 0\}$ be a $d$-dimensional L\'evy process satisfying condition (\ref{stoch_bounded}) and let the function $b(t), 0 \le t \le 1$ be as in Theorem \ref{dimension_d}. Given $x \in \R^d,$ the following are equivalent:
\begin{itemize}
\item [(a)] $ x \in C(\{X_t/b(t): t \downarrow 0\}) $ with probability 1
\item [(b)] $\int_0^1 t^{-1} \PP\{|X_t/b(t) -x |< \epsilon\}dt = \infty,\forall \epsilon > 0.$
\end{itemize}
\end{theo}
\begin{proof} \fbox{(a) $\Rightarrow$ (b)} From Lemma \ref{Kesten1} we know that (a) implies
for any $\epsilon >0$ and any $0 < r <1,$
$$\sum_{k=1}^{\infty} p(\epsilon, r, k) = \infty,$$
where $p(\epsilon, r, k):= \PP\{|X_t/b(t) -x| < \epsilon \mbox{ for some }  t \in I_k\}$ and $$I_k=[r^{k+1}, r^k[, k \ge 0.$$
Writing the integral in (b) as
$$\sum_{k=0}^{\infty} \int_{I_k} t^{-1} \PP\{|X_t/b(t) -x |< \epsilon\}dt$$
which is 
$$\ge \log(1/r) \sum_{k=0}^{\infty} \inf_{t \in {I_k}}  \PP\{|X_t/b(t) -x |< \epsilon\},$$
we see that it enough to show that given $0<\epsilon<1$,  we can find an $r_{\epsilon} \in ]0,1[$ such that for large $k$,
\be \label{kesten2a}
\PP\{|X_t/b(t) -x |< \epsilon\} \ge  p(\epsilon/4,r_{\epsilon},k+1)/2, \forall\,t \in I_k.
\ee
Consider the following stopping time
$$\tau_k := \inf\{s \ge r^{k+2}: |X_s/b(s)-x| < \epsilon/4\}.$$
Then clearly, $p(\epsilon/4,r,k+1)=\PP(\tau_k < r^{k+1}).$ Moreover, we have on the event $\{\tau_k < r^{k+1}\}:$
$$|X_{\tau_k}/b(\tau_k)| \le |x| + \epsilon/4.$$
We thus can conclude that we have on this event for any $t \in I_k,$
\begin{eqnarray*}
|X_t/b(t)-X_{\tau_k}/b(\tau_k)|&\le& |X_t- X_{\tau_k}|/b(t) + (|x|+\epsilon/4)(1-b(\tau_k)/b(t))\\
&\le&|X_t- X_{\tau_k}|/b(r^{k+1}) + \epsilon/4 + |x|(1-b(r^{k+2})/b(r^k)).
\end{eqnarray*}
If  $r=r_{\epsilon} \ge  (1-\epsilon/(4|x|+1))^{1/2} / (1-\epsilon/(8|x|+1))^{1/2}$
and we can infer from condition (\ref{50}) that there exists a  natural number $k_{\epsilon, x}$ such that  we have for $k \ge k_{\epsilon, x},$
$$b(r^{k+2})/b(r^k) \ge (1- \epsilon/(8|x|+1))r^2 \ge (1- \epsilon/(4|x|+1))$$
and it follows that
$$|X_t/b(t)-X_{\tau_k}/b(\tau_k)|\le |X_t- X_{\tau_k}|/b(r^{k+1}) + \epsilon/2, t \in I_k.$$
Next consider the stochastic  process 
$$X^*_s:= X_{\tau_k +s} -X_{\tau_k}, s \ge 0$$
which is defined on the event $\{\tau_k < r^{k+1}\}$. \\
Since $t- \tau_{k} \le r^{k}-r^{k+2}=:s_k$, we then have,
$$
\{\tau_k < r^{k+1}\} \cap \{\sup_{s \le s_k} |X^*_s| \le \epsilon b(r^{k+1})/4\} \subset \bigcap_{t \in I_k} 
\{|X_t/b(t)-x| < \epsilon\}
$$
By the strong Markov property of the L\'evy process $X_t, t \ge 0$ (see, for instance, Prop. 6 on p. 20 in \cite{Ber0}) the two left-hand events are conditionally independent on $\{\tau_k < \infty\}$  and the conditional distribution of $\{X^{*}_t: t \ge 0\}$ is equal to the $\PP$-distribution of $\{X_t: t \ge 0\}$.  We thus have,
\begin{eqnarray*}
\inf_{t \in I_k} \PP\{|X_t/b(t)-x| < \epsilon\} &\ge& \PP\{\tau_k < r^{k+1}\}\PP\{\sup_{s \le s_k} |X_s| \le \epsilon b(r^{k+1})/4\}\\
&=&p(\epsilon/4, r, k+1) \PP\{\sup_{s \le s_k} |X_s| \le \epsilon b(r^{k+1})/4\}
\end{eqnarray*}
We now need an upper bound for 
$ p_k:=\PP\{\sup_{s \le s_k} |X_s| > \epsilon b(r^{k+1})/4\}.$ By the Etemadi inequality (see, for instance, Theorem 22.5 in \cite{bil}) which also holds for L\'evy processes we have that
$$p_k \le 3 \PP\{|X_{s_k}| > \epsilon b(r^{k+1})/12\}$$
Note that $r^{k+1}/s_k= r/(1-r^2) \to \infty$ if $r \nearrow 1.$ So if $r$ is sufficiently large we have 
$$b(r^{k+1})/b(s_k) \ge r^{\rho}/(1-r^2)^{\rho} \ge 12K/\epsilon,$$ 
where in view of (\ref{stoch_bounded}) we can choose a $K >0$
so that for small enough $t,$
$$\PP\{|X_t| \ge K b(t)\}\le 1/6.$$
We can conclude that $p_k \le 1/2$ for large $k$ and we see that (\ref{kesten2a}) holds.\\[.2cm]
\fbox{(b) $\Rightarrow$ (a)} This follows since we have for any $0< r < 1,$
\begin{eqnarray*}
&&\int_0^1 t^{-1}\PP\{|X_t/b(t) -x |< \epsilon\}dt \le \sum_{k=0}^{\infty} \int_{r^{k+1}}^{r^k} t^{-1} \PP\{|X_t/b(t) -x |< \epsilon\}dt\\
&\le&\sum_{k=0}^{\infty}\log(1/r)\PP\{|X_t/b(t) -x |< \epsilon \mbox{ for some }r^{k+1} \le t < r^k\}
\end{eqnarray*}
whence condition (b) of Lemma \ref{Kesten1} is satisfied which implies that $x$ is in the cluster set. \end{proof}
\subsection{L\'evy processes without Gaussian part}
Under extra assumptions the last criterion for clustering can be further simplified as follows.
Let $A(b)$ be the symmetric non-negative definite matrix satisfying 
$$A^2(b)=\left(\int_{|y| \le b}  y_i y_j \Pi(dy)\right)_{1 \le i, j \le d}, 0 < b < 1.$$
\begin{lem}  \label{lemClu} Let $X_t, t \ge 0$ be a purely non-Gaussian L\'evy process and let $b(t), 0 \le t \le 1$ be a function as in Theorem \ref{dimension_d}.  Assume that condition (\ref{3}) is satisfied and that $\alpha_0 < \infty.$
 Then the following are equivalent:
\begin{itemize}
\item [(a)] $x \in C(\{X_t/b(t): t \downarrow 0\})$ with probability 1
\item [(b)] $\int_0^1 t^{-1} \PP\{|Y_t^{(b(t))}/b(t) - x| < \epsilon\} dt= \infty, \forall\,\epsilon >0.$
\item [(c)] $\int_0^1 t^{-1} \PP\{|\sqrt{t}A(b(t)) \eta/b(t) - x| < \epsilon\}dt= \infty, \forall\,\epsilon >0,$\\ where $\eta$ is $\mathcal{N}(0,I)$-distributed and $A(b)$ is defined as above.
\end{itemize}
\end{lem}
\begin{proof} We first prove  that (a) and (b) are equivalent.
Since we are assuming that $\alpha_0 < \infty$ it is clear that $X_t/b(t)$ is stochastically bounded as $t \downarrow 0$. Consequently, Theorem \ref{Kesten2} applies
so that (a) holds if and only if 
\be \label{52a}
\int_0^1 t^{-1} \PP\{|X_t/b(t) - x| < \epsilon\}dt= \infty, \forall \epsilon > 0
\ee
 Recall (see (\ref{Ito})) that
$X_t = t\nu(b(t)) + Y_t^{(b(t))} + Z_t^{(b(t))}, 0 < t <1,$
where we have  $t\nu(b(t))/b(t) \to 0$ as $t \to 0$ (by Lemma \ref{cond} (c) and since $t/b(t) \to 0$ as $t \to 0.$) Further noting that
$$\PP\{Z_t^{(b(t))} \ne 0\} = 1 - \exp(-t\overline{\Pi}(b(t)) \le t \overline{\Pi}(b(t)),$$
it follows that given $\epsilon >0$ there exist $0 < t_{\epsilon} <1$ such that
$$\PP\{|X_t - Y_t^{(b(t))}| \ge \epsilon b(t)\} \le t \overline{\Pi}(b(t)), 0 < t < t_{\epsilon}$$
which in turn implies by (\ref{3}) that 
$$\int_0^1 t^{-1} \PP\{|X_t - Y_t^{(b(t))}| \ge \epsilon b(t)\} dt < \infty, \epsilon >0.$$
It is easy now to see that  (\ref{52a})
holds  if and only if (b) holds.
Thus (a) and (b) are equivalent.\\[.2cm]
To see that (b) and (c) are equivalent it is enough to show that
 we have for $0 < t < 1$ and  $ \epsilon > 0$:
\be \label{eq5.2}
 \PP\{ |Y_t^{(b(t))}/b(t) -x| < \epsilon\} \le  \PP\{|\sqrt{t}A(b(t)) \eta/b(t) - x| < 2\epsilon\} + \Delta_1(t,\epsilon)\ee
and
\be \PP\{|\sqrt{t}A(b(t)) \eta/b(t) - x| < \epsilon\} \le \PP\{ |Y_t^{(b(t))}/b(t) -x| < 2\epsilon\} +\Delta_2(t,\epsilon)\ee
where $\int_0^1 t^{-1}\Delta_i(t,\epsilon) dt < \infty, i=1,2.$\\
We only prove the first inequality. The proof of the second one is then an obvious modification of the first proof.\\
Note that for any $n \ge 1,$
$$Y^{(b(t))}_t = \sum_{j=1}^n Y^{(b(t))}_{tj/n}- Y^{(b(t))}_{t (j-1)/n}=:\sum_{j=1}^n\xi_{j,t,n},$$
where the random vectors $\xi_{j,t,n}$ are i.i.d. with mean zero.\\ Moreover we have,
$\mathrm{Cov}(\xi_{j,t,n}) = \frac{t}{n} A^2(b(t))$. Since we also have $\E|\xi_{1,t,n}|^3 < \infty$
we can apply Lemma 13 in \cite{E-1} which actually is a corollary of Theorem 2 in \cite{KK}. \\
Setting $\eta_{j,t,n}:= \sqrt{t/n}A(b(t))\eta_j, 1 \le j \le n$, where $\eta_1, \ldots, \eta_n$ are independent $\mathcal{N}(0,I)$-distributed random vectors, we obtain for a   constant $C_{\epsilon}>0,$
\begin{eqnarray*}
&&\PP\{ |Y_t^{(b(t))}/b(t) -x| < \epsilon\}\\
&\le& \;\PP\{ |\sqrt{t/n}A(b(t))\sum_{j=1}^n \eta_j/b(t) - x| < 2\epsilon\} +C_{\epsilon}n \E|\xi_{1,t,n}|^3 /b(t)^3\\
&=&\PP\{ |\sqrt{t}A(b(t)) \eta_1/b(t) - x| < 2\epsilon\} +C_{\epsilon}t (n/t) \E|Y_{t/n}^{(b(t))}|^3 /b(t)^3.
\end{eqnarray*}
Recalling Lemma \ref{moment_3} and letting $n$ go to infinity, we finally find that
\begin{equation*}
\begin{split}
 \PP\{ |Y_t^{(b(t))}/b(t) -x| < \epsilon\}\le &\PP\{ |\sqrt{t}A(b(t)) \eta_1/b(t) - x| < 2\epsilon\} \\
 &\hspace{.9cm}+C_{\epsilon}t \int_0^{b(t)} |y|^3 \Pi(dy)/b(t)^3
\end{split}
\end{equation*}
which implies (\ref{eq5.2}) via Lemma \ref{cond}(c). \end{proof}
In the sequel the topological closure of a subset $B  \subset \R^d$ will always be denoted by $\mathrm{cl}(B)$.
\subsection{Proof of Theorem \ref{star}}
 It is trivial that $C(\{X_t/b(t): t \downarrow 0\})$ is closed since such cluster sets are always  closed. Note that the cluster set  $C(\{g(t): t \downarrow 0\})$ for a mapping $g: ]0,1] \to \R^d$ can be written as
$$\cap_{n \ge 1} \mathrm{cl}(\{g(s): 0 < s <1/n\}),$$
 which is closed as an intersection of closed sets in $\R^d$.\\ Next  since $\limsup_{t \downarrow 0} |X_t|/b(t)= \alpha_0 < \infty$ with probability one, we can conclude that the  cluster set $A$  has to be bounded and consequently it is compact. \\
 Moreover, there exists   an $\omega$ such that 
$\limsup_{t \downarrow 0} |X_t(\omega)|/b(t)= \alpha_0$ and at the same time $C(\{X_t(\omega)/b(t): t \downarrow 0\})=A$ as both properties hold with probability one. \\
Take a sequence $t_n = t_n(\omega) \downarrow 0$ such that $\lim_{n \to \infty} |X_{t_n}(\omega)|/b(t_n)| = \alpha_0$. Then the sequence $X_{t_n}(\omega)/b(t_n)$ is bounded and consequently there exists a subsequence $n_k=n_k(\omega) \to \infty$ such that $X_{t_{n_k}}(\omega)/b(t_{n_k})$ converges to a vector in $\R^d$ which has norm $\alpha_0$. This vector is of the form $\alpha_0 z$ with $|z|=1$ and it is in the cluster set $A.$ \\
Finally, it follows directly from Lemma \ref{lemClu}(c) that the set $A$ is symmetric about zero and also star-like at zero. Concerning the latter property we recall the following well known corollary of the classical T.W. Anderson inequality:\\  If $\eta: \Omega \to \R^d$ is a $d$-dimensional normal random vector we have for $x \in \R^d$ and $\delta >0,$
$$\PP\{ |\eta - x| < \delta\} \le \PP\{ |\eta - sx| < \delta\}, 0 \le s \le 1.$$
\subsection{Proof of Theorem \ref{star-}}
W.l.o.g we can  assume  that $\alpha_0=1$.  Our proof  is divided into three steps. In the first step we define a suitable discrete L\'evy measure $\Pi_0$ and we show that we have for any L\'evy process with characteristic triplet $(\gamma,0,\Pi_0)$  $\limsup_{t \downarrow 0} |X_t|/b(t) \le 1$ with probability 1.
In the second step we prove that the cluster set of $X_t/b(t)$ as $t\downarrow 0$ contains the set $A$. In the third step we finally show that this cluster set is also a subset of A.\\[.1cm]
{\bf Step 1} $A$ is symmetric and star-like w.r.t. zero and there is a unit vector $z_1$  in $A$.  (We are assuming $\alpha_0 =1.$) Then  $\mathcal{L}_1:=\{tz_1: |t| \le 1\}$ has to be in $A.$  Moreover, in view of  the separability of $\R^d,$ we  can write $A$ as the closure of $\mathcal{L}_1$ and at most countably many additional line segments, that is,
$$A =\mathrm{cl}(\bigcup_{j=1}^{\infty} \mathcal{L}_j),$$
where $\mathcal{L}_j = \{t z_j: |t| \le \sigma_j\},$ $z_j$ is a unit vector in $\R^d$ and $\sigma_j \in [0,1 ], j \ge 1.$ Note that $\sigma_1 =1.$ (If $A$ consists only of finitely many line segments we set $\sigma_j=0$ for large $j.$).   \\
We set $L_k:= \{\ell \in \{1,\ldots,k\}: \sigma_{\ell}^2 \ge 1/k\}, k \ge 1$ and
we denote the elements in $L_k$ by $1=j_1(k) < \ldots < j_{\ell_k}(k)$, where $\ell_k =\#L_k \ge 1$  for all $k \ge 1.$ Next we set
$$\sigma_{k, \ell} := \sigma_{j_{\ell}(k)} \mbox{ and } z_{k, l} := z_{j_{\ell}(k)}, 1 \le \ell \le \ell_k, k \ge 1.$$
Our measure $\Pi_0$ will have discrete support $\{y_{k,\ell}, 1 \le \ell \le \ell_k, k \ge k_0\} \subset D$, where again $D$ is the Euclidean unit ball and $k_0$ will be specified below. The support points are defined as follows,
$$y_{k,\ell}= b(1/a_{k,\ell})z_{k,\ell}, 1 \le \ell  \le \ell_k, k \ge k_0,$$
where $a_{k,\ell}:= h^{\leftarrow}(\exp(\exp(k^3) + \ell k)), 1 \le \ell \le \ell_k, a_{k,\ell_{k}+1}:=a_{k+1,1}, k \ge k_0$ and $k_0 \ge 2$ is  chosen  so that $a_{k_0,1} \ge e^2$.\\
Since $a_{k,\ell}, k \ge k_0, 1 \le \ell \le \ell_k$ is obviously increasing with respect to the lexicographical order on $\mathbb{N}^2$, it follows from the monotonicity of $b$ that $(k,\ell) \mapsto |y_{k,\ell}|=b(1/a_{k,\ell})$ is decreasing w.r.t to this order. \\[.1cm]
We define the discrete measure $\Pi_0$ on $D$ with  support as indicated above by setting
$$\Pi_0\{y_{k,\ell}\}= \frac{1}{b^2(1/a_{k,\ell})}\left[\frac{\sigma^2_{k,\ell}}{2h^2(a_{k,\ell})} - \frac{\sigma^2_{k,\ell+1}}{2h^2(a_{k,\ell+1})}\right], 1 \le \ell \le \ell_k, k \ge k_0,$$
where $\sigma_{k,\ell_{k}+1} :=\sigma_{k+1,1}.$\\
It then follows for any given pair $(k,\ell)$ with $k \ge k_0$ and $\ell \le \ell_k$ that
\begin{eqnarray}
\int_{|y| \le |y_{k,\ell}|} |y|^2 \Pi_0(dy) &=&  \sum_{(k',\ell') \ge (k,\ell)}|y_{k',\ell'}|^2 \Pi_0\{y_{k',\ell'}\} \nonumber\\
  &=&  \sum_{(k',\ell') \ge (k,\ell)} \left[\frac{\sigma^2_{k',\ell'}}{2h^2(a_{k',\ell'})} - \frac{\sigma^2_{k',\ell'+1}}{2h^2(a_{k',\ell'+1})}\right] =  \frac{\sigma^2_{k,\ell}}{2h^2(a_{k,\ell})} \label{5-1}
\end{eqnarray}
Applying (\ref{5-1}) with $|y_{{k_0},1}|=b(1/a_{{k_0},1})$, we see that 
$$\int_D |y|^2 \Pi_0(dy) < \infty.$$
Furthermore, we readily obtain from (\ref{5-1}) that
\be \label{5-1a}
V(b(t)) \le \int_{|y| \le b(t)}|y|^2 \Pi_0(dy) \le (2h^2 (1/t))^{-1}, 0 < t < e^{-2}.
\ee
Under the extra assumption (\ref{xy})  this already implies 
 via Theorem \ref{LIL} and Lemma \ref{extra-lemma} that with probability one,
\be \label{5-2}
\limsup_{t \downarrow 0} |X_t|/b(t) \le 1.
\ee
For the general case  we have to show directly that
$$
\int_D b^{\leftarrow}(|y|) \Pi_0(dy) < \infty.
$$
It is easy to see from the definitions of $\Pi_0$ and the function $b$ that this integral is
$$ \le \sum_{k=k_0}^{\infty}\sum_{\ell =1}^{\ell_k} \frac{1/a_{k,\ell}}{2b^2(1/a_{k,\ell})h^2(a_{k,\ell})}
=\sum_{k=k_0}^{\infty}\sum_{\ell =1}^{\ell_k}  (2 \log \log a_{k,\ell})^{-1}.$$
By the Karamata representation of the function $h$ (see, for instance, \cite{Ber0}, p.9), we have for some $c_0 > 1$ and $a_2 \ge a_1 \ge h(c_0),$
$$
\frac{a_2}{a_1} = \frac{h(h^{\leftarrow}(a_2))}{h(h^{\leftarrow}(a_1))}\le 2 \frac{h^{\leftarrow}(a_2)}{h^{\leftarrow}(a_1)}
$$
which implies that for large $k,$ $a_{k,\ell} \ge \exp(\exp(k^3)), 1 \le \ell \le \ell_k$. Recalling that $\ell_k \le k,$ we find that for some $k_1 \ge k_0,$
$$
\sum_{k=k_1}^{\infty}\sum_{\ell =1}^{\ell_k}  (2 \log \log a_{k,\ell})^{-1} \le \sum_{k=k_1}^{\infty} k^{-2} < \infty.$$
Thus (\ref{5-2}) holds in general.\\[.2cm]
{\bf Step 2} We  show that $C(\{X_t/b(t): t \downarrow 0\}) \supset A.$ \\
Since we already know that the cluster set is closed, symmetric about zero and star-like at zero, it is enough to show that we have for any fixed $j \ge 1,$
$$\sigma_j z_j \in C(\{X_t/b(t): t \downarrow 0\}) \mbox{ with probability one.}$$
It is  obviously sufficient to prove this for the $j$'s for which $\sigma_j >0.$\\ In view of Lemma \ref{lemClu} (which we can apply on account of (\ref{5-2})) this is equivalent to showing 
\be \label{5-3}
\int_0^1 t^{-1} \PP\{|\sqrt{t}\eta_t/b(t) -\sigma_j z_j| < \epsilon\}dt=\infty,\, 0 <\epsilon  < \sigma_j,
\ee
 where $\eta_t \sim \mathcal{N}(0,A^2(b(t))), 0 \le t \le 1$ and $A^2(b)$ is the $(d,d)$-matrix such that 
 $$A^2(b)_{i,j} = \int_{|y| \le b} y_i y_j \Pi_0(dy), 1 \le i, j \le d, 0 < b <1.$$
Given any vector $v \in \R^d,$ $\langle \eta_t, v \rangle$ is  a (1-dimensional) normal random variable  with mean zero and
\be \label{5-4}
\V(\langle \eta_t, v \rangle)= \langle v, A^2 (b(t)) v \rangle = \int_{|y| \le b(t)} \langle y, v \rangle^2 \Pi_0(dy), 0 < t <1.
\ee
If $k$ is large enough so that $\sigma_j^2 \ge 1/k$, we can find an index $r_k(j) \in \{1,\ldots,\ell_k\}$ for which $\sigma_{k,r_k(j)}=\sigma_j$ and $z_{k,r_k(j)}=z_j.$ From the definition of $\Pi_0$  and (\ref{5-4}) combined with the fact that $\sigma_{j'}^2 \le 1$ for all $j'$, we get  for $ t \ge 1/a_{k,r_k(j)}$
\begin{eqnarray}
\langle z_j, A^2 (b(t)) z_j  \rangle  &\ge& \sigma_j^2 [h^{-2}(a_{k,r_k(j)}) - k h^{-2}( a_{k,r_k(j) +1}) ]/2 \nonumber\\
&\ge& \sigma_j^2(1-ke^{-2k}) /(2h^2(a_{k,r_k(j)})) \label{5-5}
\end{eqnarray}
Next we define for a $\tau \in ]0,1[$ which be specified later on 
$$\tilde{a}_{k,r_k(j)}= a_{k,r_k(j)} \exp(-(\log a_{k,r_k(j)})^{\tau}/2) $$
and we set $I_k(j)=[1/a_{k,r_k(j)}, 1/\tilde{a}_{k,r_k(j)}]$.\\[.1cm]
Letting $t_k = 1/a_{k,r_k(j)}$ and $\tilde{t}_k= 1/\tilde{a}_{k,r_k(j)}$ in the proof of Theorem \ref{lower}, we can conclude that
\be \label{interval}
h(a_{k,r_k(j)})/h(\tilde{a}_{k,r_k(j)})\to 1 \mbox{ as }k \to \infty
\ee
As we also have $h(a_{k,r_k(j)})/h(a_{k,r_k(j)-1}) \to \infty$ as $k \to \infty$,  where we set
$a_{k,0}=a_{k-1,\ell_{k-1}}, k \ge 2,$ we see that  for
 sufficiently large $k,$ $\tilde{a}_{k,r_k(j)} > a_{k,r_k(j)-1}$ so that
$$ I_k(j) \subset [1/a_{k,r_k(j)}, 1/a_{k,r_k(j)-1}[.$$
If $w \in \R^d$ is a vector such that $\langle w,z_j \rangle =0$ and $t \in I_k(j),$
it  follows then that
\be \label{5-6}
\langle w, A^2 (b(t)) w  \rangle \le e^{-2k} (2h^2( a_{k,r_k(j)}))^{-1}.
\ee
Choosing an orthonormal basis  $\{w_{j,i}: 1 \le i \le d\}$  of $\R^d$ with $w_{j,1}=z_j$ and setting $\epsilon'=\epsilon/\sqrt{d}$, we can conclude that for any $t \in I_k(j)$
\begin{eqnarray*}
&&\PP\{|\sqrt{t}\eta_t/b(t) -\sigma_j z_j| < \epsilon\}\\ &\ge& \PP\left(\{ \sqrt{t}|\langle \eta_t, z_j \rangle/b(t) - \sigma_j | < \epsilon'\} \cap \bigcap_{i=2}^d \{\sqrt{t} |\langle \eta_t, w_{j,i} \rangle |/b(t)< \epsilon'\}\right) \nonumber\\
&\ge& \PP\{|\sqrt{t} \langle \eta_t, z_j \rangle/b(t) - \sigma_j | < \epsilon'\} -\sum_{i=2}^d 
\PP \{ |\langle \eta_t, w_{j,i} \rangle | \ge \epsilon'b(t)/\sqrt{t}\}. 
\end{eqnarray*}
As we have
\begin{eqnarray*}
\mathrm{Var}(\sqrt{t}\langle \eta_t, z_j\rangle/b(t))&= &t\langle z_j, A^2 (b(t)) z_j\rangle/b^2(t)\\
&\le& \int_{|y| \le b(t)}|y|^2 \Pi_0(dy)h^2(1/t)/\log \log 1/t
\end{eqnarray*}
 which goes to zero as $t \to 0$ (recall (\ref{5-1a})), one easily sees  that for large $k,$
\begin{eqnarray*}
\PP\{ |\sqrt{t}\langle \eta_t, z_j \rangle/b(t)- \sigma_j | < \epsilon'\} &=&
\PP\{ \langle \eta_t, z_j \rangle \ge  (\sigma_j - \epsilon')b(t)/\sqrt{t}\}\\
&&\hspace{.5cm} - \PP\{ \langle \eta_t, z_j \rangle \ge  (\sigma_j + \epsilon')b(t)/\sqrt{t}\}\\& \ge& \PP\{\langle \eta_t, z_j \rangle \ge  (\sigma_j - \epsilon')b(t)/\sqrt{t}\}/2.
\end{eqnarray*}
Let $\eta'$ be a standard normal random variable. Applying inequality (\ref{normal_lower}), we  obtain  from  (\ref{5-5}) and (\ref{interval})  that for $t \in I_k(j)$ and large $k,$
 \be \label{5-7}
 \begin{split}
\PP\{ \langle \eta_t, z_j \rangle \ge  (\sigma_j - \epsilon')b(t)/\sqrt{t}\}
&\ge \PP\{ \eta' \ge (1 -\epsilon')^{1/2}(2\log \log 1/t)^{1/2}\}  \\&\ge C (\log 1/t)^{-(1-\epsilon')}(\log \log 1/t)^{-1/2},
\end{split}
\ee
where $C > 0$ is a constant.\\
A similar argument using (\ref{5-6}) along with the bound 
$$\PP\{|\eta'| > t\}\le 2\exp(-t^2/2), t > 0$$
shows that for $t \in I_k(j)$ and large enough $k$,
\be \label{5-8}
\PP \{ |\langle \eta_t, w_{j,i} \rangle | \ge \epsilon'b(t)/\sqrt{t}\} \le (\log 1/t)^{-2}, 2 \le i \le d.
\ee
Combining relations  (\ref{5-7}) and (\ref{5-8}), we finally find that for $t \in I_k(j)$ and  large $k$
\be \label{5-9}
\PP\{|\sqrt{t}\eta_t/b(t) -\sigma_j z_j| < \epsilon\} \ge  C(\log 1/t)^{-(1-\epsilon')}(\log \log 1/t)^{-1/2}/4
\ee
which in turn implies that
\begin{eqnarray*}
&&\int_{I_k(j)}t^{-1}\PP\{|\sqrt{t}\eta_t/b(t) -\sigma_j z_j| < \epsilon\} dt\\
&\ge& C\log(a_{k,r_k(j)}/\tilde{a}_{k,r_k(j)}) (\log a_{k,r_k(j)})^{-(1-\epsilon')}(\log \log a_{k,r_k(j)})^{-1/2}/4\\
&\ge&C  (\log a_{k,r_k(j)})^{\tau} (\log a_{k,r_k(j)})^{-(1-\epsilon')}(\log \log a_{k,r_k(j)})^{-1/2}/4.
\end{eqnarray*}
Choosing $\tau > 1 -\epsilon'$, 
the last term converges to infinity  as $k$ goes to infinity. Thus (\ref{5-3}) holds which means that $\sigma_j z_j \in C(\{X_t/b(t): t \downarrow 0\}).$\\[.2cm]
{\bf Step 3}  We show that $x \not \in A$ implies  $x \not \in C(\{X_t/b(t): t \downarrow 0\}).$\\ Set $\epsilon := \mathrm{dist}(x, A)/2.$ This is a positive number since $A$ is closed.\\
In view of Lemma \ref{lemClu}  it is sufficient to prove for this choice of $\epsilon,$
\be \label{5-10}
\int_0^1 t^{-1} \PP\{|\sqrt{t}\eta_t /b(t) - x| < \epsilon\}dt < \infty,
\ee
where, as in Step 2, $\eta_t \sim \mathcal{N}(0,A^2(b(t))), 0 < t <1.$\\
Consider the intervals 
$$J_{k,\ell}:= [1/a_{k,\ell}, 1/a_{k,\ell-1}[, 1 \le \ell \le \ell_k, k \ge 2.
$$
Recall  that  $a_{k,0}=a_{k-1,\ell_{k-1}}, k \ge 2.$
Similarly as in (\ref{5-5}) and (\ref{5-6}), we then can conclude that for $t \in J_{k,\ell},$
\begin{eqnarray}
\langle z_{k,\ell}, A^2(b(t)) z_{k,\ell}\rangle \le (\sigma^2_{k,\ell}+e^{-2k})(2h^2(1/t))^{-1} \label{5-11}\\
\langle w, A^2(b(t)) w\rangle \le e^{-2k}(2h^2(1/t))^{-1} \mbox{ if }\langle w, z_{k,\ell}\rangle =0.
\label{5-12}
\end{eqnarray}
Furthermore, we have by definition of $\epsilon$ that 
\be \label{5-13}
\begin{split}
 \PP\{|\sqrt{t}\eta_t/b(t) - x| < \epsilon\} &\le \PP\{\mathrm{dist}(\sqrt{t}\eta_t/b(t), A) >\epsilon\}\\
 & \le \PP\{\mathrm{dist}(\sqrt{t}\eta_t/b(t), \mathcal{L}_{k,\ell}) >\epsilon\},
\end{split}
\ee
where $\mathcal{L}_{k,\ell}=\{t z_{k,\ell}: |t| \le\sigma_{k,\ell}\}, 1 \le \ell \le \ell_k, k \ge k_0.$\\
Let $\{w_{k,\ell, i}: 1 \le i \le d\}$ be an orthonormal basis of $\R^d$ with $w_{k,\ell, 1}=z_{k,\ell}.$  Writing 
$$\eta_t = \langle \eta_t, z_{k,\ell}\rangle z_{k,\ell} + \sum_{i=2}^d  \langle \eta_t, w_{k,\ell, i}\rangle w_{k,\ell, i},$$
it is easy to see that
\be
\begin{split}
 \PP\{\mathrm{dist}(\sqrt{t}\eta_t/b(t), \mathcal{L}_{k,\ell}) >\epsilon\} &\le \PP\{|\langle \eta_t, z_{k,\ell}\rangle | > (\sigma_{k,l}
  +\epsilon')b(t)/\sqrt{t}\}\\& \quad +  \sum_{i=2}^d \PP\{|\langle \eta_t, w_{k,\ell, i}\rangle| \ge \epsilon'b(t)/\sqrt{t}\},
\end{split}
\ee
 where, as in Step 2, $\epsilon'=\epsilon/\sqrt{d}.$\\
 Using  the same exponential inequality for the normal distribution as in (\ref{5-8}) along with relations (\ref{5-11}) and (\ref{5-12}), we can conclude that for large enough $k$ and $t \in J_{k,\ell}$,
 \be 
  \PP\{\mathrm{dist}(\eta_t, \mathcal{L}_{k,\ell}) >\epsilon\} \le 2d (\log 1/t)^{-1-\epsilon'}
  \ee
 This  implies via (\ref{5-13}) that for some $k_2 \ge k_0,$
 \begin{eqnarray*}
&&\int_0^{1/a_{k_2,0}} t^{-1} \PP\{|\sqrt{t}\eta_t/b(t) - x| < \epsilon\} dt\\ &=& \sum_{k=k_2}^{\infty} \sum_{\ell =1}^{ \ell_k} \int_{J_{k,\ell}}t^{-1} \PP\{|\sqrt{t}\eta_t/b(t) - x| < \epsilon\} dt\\
 &\le& 2d \int_0^{e^{-1}} t^{-1} \log(1/t)^{-1-\epsilon'}dt < \infty.
\end{eqnarray*}
Thus (\ref{5-10}) holds and Theorem \ref{star-} has been proven. \qed\\[.3cm]
{\bf Acknowledgement}. A large part of this work was done during two visits at the Australian National University in April 2017 and April 2018. The author would like to thank Professors B. Buchmann and R. Maller for making these visits possible and for very stimulating discussions about L\'evy processes.

\end{document}